\documentclass[a4paper]{amsart}
\usepackage{amssymb,amsfonts}
\usepackage{mathabx}
\usepackage[all,arc]{xy}
\usepackage{enumerate}
    \usepackage{sfmath}
\usepackage{tikz}
\usetikzlibrary{arrows,snakes,backgrounds}
\usepackage{color}   %May be necessary if you want to color links
\usepackage{marginnote}
\usepackage{tikz-cd}
\usetikzlibrary{positioning}
\usepackage{enumitem}
\usepackage{hyperref}
\hypersetup{
    colorlinks=true, %set true if you want colored links
    linktoc=all,     
    citecolor=black,
    filecolor=black,
    linkcolor=blue,
    urlcolor=black
}
\usepackage{cleveref}% Has to be loaded after hyperref
\usetikzlibrary{arrows}

\usepackage{mathabx,epsfig}

\newtheorem{thm}{Theorem}[section]
\newtheorem{cor}[thm]{Corollary}

\newtheorem{lem}[thm]{Lemma}
\newtheorem{conj}[thm]{Conjecture}
\newtheorem*{conj*}{Conjecture}

\newtheorem*{openproblem*}{Problem}
\newtheorem*{quest*}{Question}
\newtheorem*{problem*}{Problem}
\newtheorem*{claim*}{Claim}

\theoremstyle{definition}

\theoremstyle{remark}
\newtheorem{rem}[thm]{Remark}

\newcommand{\bC}{\mathbb{C}}

\newcommand{\bQ}{\mathbb{Q}}
\newcommand{\bR}{\mathbb{R}}

\newcommand{\bZ}{\mathbb{Z}}

\newcommand\Diff{\mathrm{Diff}}

\newcommand\BDiff{\mathrm{BDiff}}
\newcommand\dDiff{\mathrm{Diff}^{\delta}}

\newcommand\BdDiff{\mathrm{BDiff}^{\delta}}

\newcommand{\BdH}{\mathrm{B}\text{\textnormal{Homeo}}^{\delta}}

\makeatother
\addtocontents{toc}{\protect\setcounter{tocdepth}{1}}
\makeatletter
\let\c@equation\c@thm
\makeatother
\numberwithin{equation}{section}
\title[]{On flat manifold bundles and the connectivity of Haefliger's classifying spaces}

\author{Sam Nariman}
\address{Department of Mathematics\\
  Purdue University\\
150 N. University Street\\
West Lafayette, IN 47907-2067\\
}
\email{snariman@purdue.edu}

\begin{document}
\begin{abstract}
We investigate a conjecture due to Haefliger and Thurston in the context of foliated manifold bundles. In this context,  Haefliger-Thurston's conjecture predicts that every $M$-bundle over a manifold $B$ where $\text{dim}(B)\leq \text{dim}(M)$ is cobordant to a flat $M$-bundle. In particular, we study the bordism class of flat $M$-bundles over low dimensional manifolds, comparing a finite dimensional Lie group $G$ with $\Diff_0(G)$. 
%and localizing the holonomy of flat M-bundles to be supported in a ball.
%We first consider PL foliations of codimension $2$ and Haefliger-Thurston's conjecture in this case says that  the classifying space $\overline{\mathrm{B}\Gamma}_2^{PL}$ is $4$-connected. We show that this classifying space  is at least $3$-connected and $\pi_4(\overline{\mathrm{B}\Gamma}_2^{PL})\otimes \bF_p=0$ for all prime $p$. As a consequence, we answer a question of Epstein regarding the simplicity of the identity component of PL homeomorphisms in dimension $2$. For smooth foliations, We investigate low homological consequences of Haefliger-Thurston's conjecture in the context of foliated manifold bundles.
%We investigate low homological consequences of a conjecture due to Haefliger and Thurston in the context of foliated manifold bundles. In particular,  Haefliger-Thurston's conjecture predicts that every $M$-bundle over a manifold $B$ where $\text{dim}(B)\leq \text{dim}(M)$ is cobordant to a flat $M$-bundle. We prove this prediction ``up to torsion" when $B$ is a $3$-manifold and  for $\text{dim}(M)\nequiv 1 \bmod 4$. We also consider the case of PL foliations of codimension $2$ and Haefliger-Thurston's conjecture in this case says that  the classifying space $\overline{\mathrm{B}\Gamma}_2^{PL}$ is $4$-connected. We show that this classifying space  is at least $3$-connected and $\pi_4(\overline{\mathrm{B}\Gamma}_2^{PL})\otimes \bF_p=0$ for all prime $p$. As a consequence, we answer a question of Epstein regarding the simplicity of the identity component of PL homeomorphisms in dimension $2$. 
\end{abstract}
\maketitle
\section{Introduction}
To build a classifying space for codimension $n$ foliations, Haefliger considered a more relaxed structure known as codimension $n$  Haefliger structures and built a classifying space $\mathrm{B}\Gamma^{r,+}_n$ for them (e.g. see  \cite{haefliger1971homotopy, bott1972lectures}) where $\Gamma^{r,+}_n$ is the etale groupoid of germs of local orientation preserving $C^r$-diffeomorphisms of $\bR^n$. There is a natural map 
\[
\nu\colon \mathrm{B}\Gamma^{r,+}_n\to \mathrm{BGL}^+_n(\bR),
\]
which classifies the oriented normal bundle to the $C^r$-Haefliger structures of codimension $n$. If we drop the regularity $r$, we mean the smooth case. 

Studying the homotopy type of the classifying space $\mathrm{B}\Gamma^{r,+}_n$ has deep consequences in foliation theory. In particular, it implies integrability of plane fields up to homotopy in a range of dimensions because of the h-principle theorems due to Thurston about Haefliger's structures (\cite{MR0370619,MR0425985}).  Let $\overline{\mathrm{B}\Gamma}^{r}_n$ denote the homotopy fiber of $\nu$. This space classifies those Haefliger structures with the trivial normal bundle. Haefliger used Gromov-Phillips'  theorem in \cite{haefliger1971homotopy} to show that $\overline{\mathrm{B}\Gamma}^{r}_n$ is at least $n$-connected. Thurston first proved (\cite{thurston1974foliations})  that the identity component of the smooth diffeomorphism group of any compact manifold is a simple group and used it to show that that $\overline{\mathrm{B}\Gamma}_n$ is $(n+1)$-connected and shortly after Mather (\cite[Section 7]{MR0356129}) proved the same statement for $\overline{\mathrm{B}\Gamma}^{r}_n$ when $r\neq \text{dim}(M)+1$. 

Haefliger introduced and calculated {\it differentiable} cohomology of $\overline{\mathrm{B}\Gamma}_n$  in \cite{MR660658} and showed that it vanishes up to  degree $2n$. And he speculated (\cite[Section 6]{MR660658}) the possibility that  $\overline{\mathrm{B}\Gamma}_n$ might be $2n$-connected.  Thurston also stated (\cite{thurston1974foliations}) this range of connectivity for $\overline{\mathrm{B}\Gamma}^r_n$ as a conjecture. Using Mather-Thurston's theory (see \cite{mather2011homology, nariman2020thurston}), one could equivalently state this conjecture in the context of manifold bundles. Let $\Diff^r(M)$ denote the group of $C^r$-orientation preserving diffeomorphisms of a smooth manifold $M$ with the $C^r$-Whitney topology. We decorate it with superscript $\delta$ and subscript $c$ if we consider the same group with discrete topology and  its subgroup of compactly supported diffeomorphisms respectively. The identity homomorphism between the groups $\Diff_c^r(M)^{\delta}\to \Diff_c^r(M)$ induces the map between classifying spaces
   \begin{equation}\label{eta}
   \eta: \BDiff_c^r(M)^{\delta}\to \BDiff_c^r(M).
   \end{equation}

   \begin{conj}[Haefliger-Thurston]\label{HT}
Let $M$ be an oriented closed manifold. The  map $\eta$ is a homology isomorphism in degrees less than or equal to $\text{dim}(M)$ and is a surjection on homology in degree $\text{dim}(M)+1$. 
\end{conj}
Geometrically, this conjecture is equivalent to saying that for every smooth $M$-bundle $M\to E\to B$ where $B$ is a manifold and $\text{dim}(B)\leq \text{dim}(M)$, there exists a bordism  $W$ from $B$ to another manifold $B'$ and an $M$-bundle $M\to K\to W$ such that when it is restricted to $B$, it is isomorphic to $E\to B$ and when it is restricted to $B'$, it is a flat $M$-bundle i.e. it is induced by a representation $\pi_1(B')\to \Diff^\delta(M)$. 
\begin{rem}
This conjecture can be stated for $C^r$-diffeomorphisms for all regularities $r$. In fact for $r=0$, it is a consequence of Mather's theorem (\cite{MR0288777}) that $\eta$ induces a homology isomorphism in all degrees and the same holds for $r=1$  as a consequence of Tsuboi's remarkable theorem (\cite{tsuboi1989foliated}).
\end{rem}

This conjecture in the smooth category seems to be out of reach at this point but in this paper, we want to investigate certain low dimensional predictions of this conjecture. In particular, we consider certain cases to investigate surjectivity and injectivity of $\eta_*$ in low homological degrees.
\begin{rem}
Using the work of Peter Greenberg (\cite{MR1200422}) we shall prove in a separate paper \cite{Nariman}, new connectivity results for the curious case of PL-foliations in codimension $2$. Then  using the version of Mather-Thurston's theorem for PL homeomorphisms (\cite{nariman2020thurston}) due to the author,  we prove the perfectness of PL homeomorphisms of surfaces that are isotopic to the identity which answers a question (\cite[Section 3]{MR267589}) of Epstein in dimension $2$. 
\end{rem}

\subsection{On the surjectivity part of \Cref{HT} in low degrees.} The first nontrivial homological degree is the case $*=3$. The fact that $\eta_*$ induces an isomorphism in $*=1$ and it is surjective for $*=2$ is a consequence of Thurston's theorem (\cite{thurston1974foliations}) that the identity component $\dDiff_0(M)$ is a simple group for any closed smooth manifold $M$. Let $\overline{\BDiff(M)}$ denote the homotopy fiber of the map $\eta$. Then Thurston's simplicity result implies that $H_1(\overline{\BDiff(M)};\bZ)=0$ which in particular leads to surjectivity of $\eta_*$ for $*\leq 2$. 

When the dimension of $M$ is $2$ or $3$, we know a lot about the homotopy type of $\Diff_0(M)$. In particular, in dimension $2$, if a surface $\Sigma_g$ has genus $g$ larger than $1$, then $\Diff_0(\Sigma_g)$ is contractible (\cite{earle1969fibre}). So the surjectivity of $\eta_*$ for the identity component is obvious. In \cite[Theorem 3.17]{nariman2015stable} for the entire group $\Diff(\Sigma_g)$, we also proved that for the case of surfaces, the natural map
\[
\eta_*\colon H_3(\BdDiff(\Sigma_g);\bZ)\to H_3(\BDiff(\Sigma_g);\bZ),
\]
is surjective. It is known that for  $g>3$, the group $H_3(\BDiff(\Sigma_g);\bZ)$ is torsion but nonetheless we have surjectivity in degree $3$ with integral coefficients. 
\begin{rem}
For a connected finite dimensional Lie group $G$, the group $H_3(\mathrm{B}G;\bZ)$ is finite. Hence by \cite[Corollary  2 section 3]{milnor1983homology} (see also \cite[Lemma 6]{milnor1983homology}) the map $H_3(\mathrm{B}G^{\delta};\bZ)\to H_3(\mathrm{B}G;\bZ)$ is surjective.
\end{rem}

Therefore, the first nontrivial case of \Cref{HT} for the identity component $\Diff_0(M)$  is  when $\text{dim}(M)=3$ and $*=3$. As a consequence of the resolution of the generalized Smale's conjecture (\cite{hatcher1983proof, MR0420620, MR0448370, gabai2001smale, MR2976322, MR3024309, bamler2019ricci, bamler2019diffeomorphisms}) we know about the homotopy type of $\Diff_0(M)$ when $M$ is a geometric $3$-manifold. In particular, in many cases we have $H_3(\BDiff_0(M);\bQ)=0$ but one interesting example is the case $M\cong S^1\times S^2$ where Hatcher's theorem (\cite{hatcher1981diffeomorphism}) implies that $H_3(\BDiff_0(S^1\times S^2);\bQ)=\bQ$. In \Cref{H_3}, we prove that the natural map
\[
\eta_*\colon H_3(\BdDiff_0(S^1\times S^2);\bZ)\to H_3(\BDiff_0(S^1\times S^2);\bZ),
\]
is surjective. And with rational coefficients, we prove the following more general result.
\begin{thm}\label{H3} Let $M$ be a closed manifold such that $\text{dim}(M)\nequiv 1 \bmod 4$, then the map 
\[
\eta_*\colon H_3(\BdDiff_0(M);\bQ)\to H_3(\BDiff_0(M);\bQ),
\]
is surjective.
\end{thm}
\begin{rem}
We expect that one might be able to drop the hypothesis $\text{dim}(M)\nequiv 1 \bmod 4$. As we shall see in \Cref{H_3},  an affirmative answer to a question posed by Vogt (\cite[Problem F.2.1]{MR1271828}) is a step towards making this hypothesis unnecessary. 
\end{rem}
%\begin{rem}
%In \cite[Theorem 3.17]{nariman2015stable}, we proved that for the case of surfaces the natural map
%\[
%\eta_*\colon H_3(\BdDiff(\Sigma_g);\bZ)\to H_3(\BDiff(\Sigma_g);\bZ),
%\]
%is surjective. It is known that for  $g>3$, the group $H_3(\BDiff(\Sigma_g);\bZ)$ is torsion but nonetheless we have surjectivity in degree $3$ with integral coefficients. 
%\end{rem}
\subsection{On the injectivity part of \Cref{HT} in low degrees} The injectivity part specially for regularities seems to be notoriously difficult. We instead try to investigate some of its predictions. The only results known in this direction is due to Tsuboi in low regularities (see \cite{MR877328, tsuboi1989foliated}). His work implies that 
\[
\eta\colon \mathrm{B}\Diff^{r,\delta}(M)\to \BDiff^r(M),
\]
induces a homology isomorphism in all degrees for $r=1$ and in general it induces a homology isomorphism in degrees less $m$ where $r<\frac{n+1}{m}-1$. In particular, for $C^{\infty}$-diffeomorphisms, it is still open that whether 
\[
\eta_*\colon H_2(\mathrm{B}\Diff^{\delta}(M);\bZ)\to H_2(\BDiff(M);\bZ),
\]
is injective. To prove these injectivity results, Tsuboi (\cite{MR877328, tsuboi1989foliated}) extensively studied the vanishing of $H_*(\overline{\BDiff(M)};\bZ)$ in low homological degrees. One presumably easier question would be whether $H_k(\overline{\BDiff^r(M)};\bZ)$ is independent of $M$ for $k\leq \text{dim}(M)$ or for a codimension zero embedding $N\hookrightarrow M$, whether the natural map 
\[
H_k(\overline{\BDiff^r(N, \partial)};\bZ)\to H_k(\overline{\BDiff^r(M, \partial)};\bZ),
\]
is surjective for $k\leq \text{dim}(M)$. We shall follow this perspective in low dimensions. There are natural ways to build cycles in group homology of diffeomorphism groups and use Haefliger-Thurston's conjecture to predict that they are trivial. Proving these predictions, in each separate example seem to already be very nontrivial and we think they are interesting on their own. For example, Tsuboi used the flow of vector fields to build abelian cycles (\cite{MR928395, 928395}) in diffeomorphism groups and proved their triviality for codimension one foliations and conjectured the same in all dimensions. 
%The injectivity even low degrees seems very hard and it is not even known in degree $2$ unless the regularity of diffeomorphisms are very low (see \cite{tsuboi1989foliated}). 

%Recall that the only known classes in $H^*(\BdDiff(M);\bR)$ that do not come from $H^*(\BDiff(M);\bR)$ are induced by the fiber integration of Godbillon-Vey classes (see \cite{pittie1976characteristic} or  \cite[Section 10]{bott1972lectures} for definitions of Godbillon-Vey classes) which live in degree $\text{dim}(M)+1$. Given the lack of nontrivial group cocycles in degree $\text{dim}(M)$ and lower, the injectivity part of \Cref{HT} predicts that many natural group cycles for $\dDiff(M)$ should bound. For example, Tsuboi considered the following cycles (\cite{MR928395}). Let $\xi$ be a compactly supported vector field on $M$. Then the flow of $\xi$ induces a group homomorphism
%\[
%\phi\colon \mathrm{B}\bR^{\delta}\to \BdDiff_c(M),
%\]
%where $\Diff_c(M)$ denotes the compactly supported diffeomorphisms. Note that we have $H_*( \mathrm{B}\bR^{\delta};\bZ)\cong \bigwedge_{\bQ}^*\bR$. Hence, we obtain many cycles in $H_*(\BdDiff(M);\bZ)$. Tsuboi asked whether these cycles bound and showed that in fact $\phi_*$ is a trivial map for the case where $M\cong \bR$. And he posed a question in \cite{928395} that whether the same holds for all manifolds. 

Here we consider another source of natural cycles in group homology of diffeomorphisms by letting a Lie group $G$ act on itself. Let $G$ be a connected Lie group. The group homology or the homology of the space $\mathrm{B}G^{\delta}$ has been extensively studied for certain Lie groups since it is related to Milnor's conjecture (\cite{milnor1983homology}) and also to scissors congruence (see \cite{MR928063, MR982639, MR856093} and references therein).  It is a deep result of Sah-Wagoner (\cite[Theorem 1.28]{MR646087}) that for any connected Lie group $G$, the second group homology $H_2(G,\bZ)$ has a quotient group $K_2(\bC)^+$, the positive part of the second K-group of $\bC$ which in particular is a $\bQ$-vector space of dimension equal to the continuum. So if we consider the natural map 
\[
\mathrm{B}G^{\delta}\to \BDiff^{r,\delta}_0(G),
\]
we can map nontrivial cycles in $H_*(\mathrm{B}G^{\delta};\bZ)$ to $H_*(\BDiff^{r,\delta}_0(G);\bZ)$. As we shall see, in low regularity, we have

%For example it is an unpublished result of Cheeger-Simons that $H_3(\text{SO}(4);\bZ)\otimes \bQ$ is a nontrivial vector space of at least countable dimension. 
% On the one hand, Milnor proved (\cite{milnor1983homology}) that for a compact or complex semi-simple Lie group $H_*(\mathrm{B}G^{\delta};\bQ)\to H_*(\mathrm{B}G;\bQ)$ is a trivial map. Given the commutative diagram
%\[
% \begin{tikzpicture}[node distance=1.5cm, auto]
%  \node (A) {$ H_*(\mathrm{B}G^{\delta};\bQ)$};
%  \node (B) [right of=A, node distance=3cm] {$H_*(\BDiff^{r,\delta}_0(G);\bQ)$};
%  \node (C) [below of=A]{$H_*(\mathrm{B}G;\bQ)$};
%  \node (D) [right of=C, node distance=3cm]{$H_*(\BDiff^r_0(G);\bQ),$};
%  \draw [->] (A) to node {$$}(B);
%    \draw [->] (A) to node {$$}(C);
%  \draw [->] (C) to node {$$}(D);
%  \draw [->] (B) to node {$$}(D);
%\end{tikzpicture}
%\]
%and \Cref{HT}, in homological dimension less than $\text{dim}(G)+1$, we expect that the top map is trivial. Note that since \Cref{HT} is known for $r=0,1$ in all homological degree, we already conclude the following in low regularity.
\begin{thm}\label{lowreg}
For all compact or complex semi-simple Lie groups, the map
\[
 H_*(\mathrm{B}G^{\delta};\bQ)\to H_*(\BdH_0(G);\bQ),
\]
is a trivial map. The same holds for $\Diff^1_0(G)$.
\end{thm}
Geometrically, a nontrivial element in  $H_n(\mathrm{B}G^{\delta};\bQ)$ can be represented by a flat $G$-bundle over a manifold $M^n$ that cannot be extended to a flat $G$ bundle over  a manifold $W^{n+1}$ such that $\partial W=M$. So the above theorem says (up to torsion) any such bundle bounds a flat $\text{Homeo}_0(G)$-bundle. 
However, in higher regularities, we could only show that
\begin{thm}
 The map
\[
H_2(\mathrm{B}G^{\delta};\bQ)\to H_2(\BDiff^{r,\delta}_0(G);\bQ),
\]
is a trivial for any noncompact Lie group, abelian Lie group and also in a special case $G=\text{\textnormal{SU}}_2$.
\end{thm}
\begin{thm}
For a complex semisimple Lie group $G$, the map
\[
(\alpha_r)_*\colon H_*(\mathrm{B}G^{\delta};\bQ)\to H_*(\BDiff^{r,\delta}_0(G);\bQ),
\]
is trivial for $*\leq \text{dim}(G/K)+1$ where $K$ is a maximal compact Lie subgroup.

\end{thm}
However we consider continuous group cohomology, as we shall see in \Cref{flat}, we can drop the hypothesis on the degree.
\subsection*{Acknowledgement} The author was partially supported by  NSF CAREER Grant DMS-2239106, NSF DMS-2113828 and Simons Foundation (855209, SN). The author  thanks Gael Meigniez for the discussion about Haefliger-Thurston's conjecture. The author also thanks Mike Freedman for the discussion about \Cref{flat}. He is also grateful to  Jonathan Bowden,  Sander Kupers and S\o ren Galatius for their comments.

\section{Making manifold bundles flat over $3$-manifolds up to bordism}\label{H_3}
\subsection{Background} \label{back}Let $\Diff^r(M)$ denote the group of $C^r$-diffeomorphisms of a smooth manifold $M$ with the $C^r$-Whitney topology. We decorate it with superscript $\delta$ and subscript $c$ if we consider the same group with discrete topology and the  subgroup of compactly supported diffeomorphisms respectively. If we drop $r$, we mean smooth diffeomorphisms. Recall that one can associate to any topological group $G$, the {\it classifying} space $\mathrm{B}G$ which classifies principal bundles whose group structures are $G$. The identity homomorphism $\Diff_c^r(M)^{\delta}\to \Diff_c^r(M)$ induces the map between classifying spaces
   \[
   \eta: \BDiff_c^r(M)^{\delta}\to \BDiff_c^r(M).
   \]

 Thurston in fact studied $\overline{\BDiff_c^r(M)}$ which is the homotopy fiber of the map $\eta$. This space classifies foliated trivial $M$-bundles. One can give a semi-simplicial model for $\overline{\BDiff_c^r(M)}$ where the set of $k$-simplicies are the set of foliations on the trivial bundle $\Delta^k\times M\to \Delta^k$ that are transverse to the fibers and whose holonomies lie in $\Diff^r_c(M)$. 
    Thurston proved an h-principle type theorem that the geometric object  $\overline{\BDiff_c^r(M)}$ is homology isomorphic to the space of {\it compactly supported} sections of a bundle over $M$ which is more amenable to tools from algebraic topology. To explain this bundle, let $\Gamma_n^r$ denote the topological groupoid whose space of objects is $\bR^n$ and space of morphisms is given by germs of $C^r$-diffeomorphisms between two points in $\bR^n$ with a sheaf topology (see \cite{haefliger1971homotopy}). There is a natural map
   \[
 \nu\colon  \mathrm{B}\Gamma_n^r\to \mathrm{B}\text{GL}_n(\bR),
   \]
    induced by the derivatives of germs. Let $\overline{\mathrm{B}\Gamma_n^r}$ denote the homotopy fiber of the map $\nu$. Let $\tau_M\colon M\to  \mathrm{B}\text{GL}_n(\bR)$ be the map that classifies the tangent bundle and $\tau_M^*(\nu)$ be the bundle over $M$ induced by the pullback of the map $\nu$ via $\tau_M$. Fix a section $s_0$ of this bundle. For any other section $s$, we can define the support of $s$ to be the set of points $x\in M$ where $s(x)\neq s_0(x)$. Let $\text{Sect}_c(\tau_M^*(\nu))$ be the subspace of compactly supported sections of the bundle $\tau_M^*(\nu)\to M$. Alternatively, it is homotopy equivalent to the space of lifts $\alpha$ of the map $\tau_M$ in the diagmram
    \[
 \begin{tikzpicture}[node distance=1.8cm, auto]
  \node (A) {$ M$};
  \node (B) [right of=A, node distance=3cm] {$\mathrm{B}\text{GL}_n(\bR),$};
  \node (C) [above of= B ] {$\mathrm{B}\Gamma_n^r$};  
   \draw [->] (A) to node {$\tau_M$}(B);
  \draw [->] (C) to node {$\nu$}(B);
  \draw [->] (A) to node {$\alpha$}(C);
\end{tikzpicture}
\]
    Mather-Thurston's theorem \cite{mather2011homology} says that there is a map $
 \overline{\BDiff_c^r(M)}\to \text{\textnormal{Sect}}_c(\tau_M^*(\nu)),$ which induces a homology isomorphism. 

Note that if $M$ is parallelizable, the map $\tau_M$ is null-homotopic, so the space of lifts $\alpha$ is homotopy equivalent to the space of maps $\text{Map}(M, \overline{\mathrm{B}\Gamma}_n)$.

We shall use the following standard fact about homotopy groups of mapping spaces. Suppose $M$ is an $n$-dimensional connected manifold and $X$ is an $(n+k)$-connected CW complex. Note that $\pi_i(\text{Map}(M, X))=[S^i, \text{Map}(M, X)]$ where the bracket denotes pointed homotopy classes. Now using adjunction and the cellular approximation imply the following fact.
Then $\pi_i(\text{Map}(M, X))=0$ for $i\leq k$ and $\pi_{k+1}(\text{Map}(M, X))=\pi_{n+k+1}(X)$. \subsection{Proof of \Cref{H3}}
Let $\Diff_0(M)$ be the identity component of $\Diff(M)$. Sometimes it is easier to work with the identity component first. Using the short exact sequence $\Diff_0(M)\to \Diff(M)\to \pi_0(\Diff(M))$ and comparison of Hochschild-Serre spectral sequences, it is easy to see that Haefliger-Thurston's conjecture can be deduced from the same statement for the map $\eta\colon \BdDiff_0(M)\to \BDiff_0(M)$.

Now suppose that we have an $M$-bundle over a $3$-manifold whose bundle group structure is $\Diff_0(M)$. To make this bundle flat up to bordism (see also \cite{freedman2020controlled}) we want to see whether the map
\[
\eta_*\colon H_3(\BdDiff_0(M);\bZ)\to H_3(\BDiff_0(M);\bZ),
\]
is surjective. Let us first consider an interesting nontrivial case where  $M\cong S^1\times S^2$. Using the perfectness of $\dDiff_0(M)$, We know that the $\overline{\BDiff_0(M)}$, homotopy fiber of $\eta$ has vanishing first homology for any manifold $M$. And by Mather-Thurston's theorem (\cite{mather2011homology}), since $3$-manifolds $M$ are parallelizable, $\overline{\BDiff_0(M)}$ is homology isomorphic to the space of maps $\text{Map}(M, \overline{\mathrm{B}\Gamma}_3)$. But by Thurston's theorem (\cite{thurston1974foliations}) $\overline{\mathrm{B}\Gamma}_n$ is at least $(n+1)$-connected. Therefore, the space $\text{Map}(M, \overline{\mathrm{B}\Gamma}_3)$ is simply connected, so we have 
\begin{equation}\label{MT}
H_2(\text{Map}(M, \overline{\mathrm{B}\Gamma}_3);\bZ)\cong \pi_2(\text{Map}(M, \overline{\mathrm{B}\Gamma}_3))\cong \pi_5(\overline{\mathrm{B}\Gamma}_3),
\end{equation}
where the first isomorphism is by the Hurewicz theorem and the second is given by the above standard fact about mapping spaces. On the other hand, by the same argument $H_2(\overline{\BDiff_c(\bR^3)};\bZ)$ is also isomorphic to $\pi_5(\overline{\mathrm{B}\Gamma}_3)$. Therefore, for all embeddings of an open disk $\bR^3\hookrightarrow M$, the map
\[
\overline{\BDiff_c(\bR^3)}\to \overline{\BDiff_0(M)},
\]
induces an isomorphism on $H_2$. Note that $\BDiff_0(M)$ is simply connected, so to prove that $\eta_*$ is surjective on $H_3$, it is enough to prove the following $d_3$ differential in the Serre spectral sequence for $\overline{\BDiff_0(M)}\to \BdDiff_0(M)\to \BDiff_0(M)$ is trivial,
\[
d_3\colon H_3(\BDiff_0(M);\bZ)\to H_2(\overline{\BDiff_0(M)};\bZ).
\]
For many geometric $3$-manifolds for which we know the homotopy type of $\Diff_0(M)$ by the generalized Smale's conjecture, $H_3(\BDiff_0(M);\bZ)=0$. Hence, for those $3$-manifolds, the same statement follows directly from the generalized Smale's conjecture. But for the case $M\cong S^1\times S^2$ where $H_3(\BDiff_0(M);\bZ)=Z$, we show that the above differential vanishes by comparison of spectral sequences. 

First let $N\cong [0,1]\times S^2$ be a submanifold of $M$ such that $N\hookrightarrow M$ induces an isomorphism on $\pi_2$. Again by Hatcher's theorem (see listed of equivalent statements in \cite[Appendix]{hatcher1983proof}) the group $\Diff_0(N,\text{rel }\partial)$ is homotopy equivalent to the base point component of the loop space $\Omega \text{SO}_3$. Therefore, the map 
\[
\BDiff_0(N,\text{rel }\partial)\to \BDiff_0(M),
\]
induces an isomorphism on $H_3$. Since we have the commutative diagram 
\[
 \begin{tikzpicture}[node distance=1.5cm, auto]
  \node (A) {$ H_3(\BdDiff_0(N,\text{rel }\partial);\bZ)$};
  \node (B) [right of=A, node distance=4cm] {$ H_3(\BDiff_0(N,\text{rel }\partial);\bZ)$};
  \node (C) [below of=A]{$ H_3(\BdDiff_0(M);\bZ)$};
  \node (D) [right of=C, node distance=4cm]{$ H_3(\BDiff_0(M);\bZ),$};
  \draw [->] (A) to node {$$}(B);
    \draw [->] (A) to node {$$}(C);
  \draw [->] (C) to node {$$}(D);
  \draw [->] (B) to node {$\cong$}(D);
\end{tikzpicture}
\]
it is enough to show that the horizontal map is surjective. But now by capping off one of the sphere boundary components of $N$, we obtain an embedding $N\hookrightarrow D^3$ which induces a commutative diagram up to homotopy
\[
 \begin{tikzpicture}[node distance=1.5cm, auto]
  \node (A) {$ \BdDiff_0(N,\text{rel }\partial)$};
  \node (B) [right of=A, node distance=4cm] {$ \BDiff_0(N,\text{rel }\partial)$};
  \node (C) [below of=A]{$ \BdDiff_0(D^3,\text{rel }\partial)$};
  \node (D) [right of=C, node distance=4cm]{$  \BDiff_0(D^3,\text{rel }\partial).$};
  \draw [->] (A) to node {$$}(B);
    \draw [->] (A) to node {$$}(C);
  \draw [->] (C) to node {$$}(D);
  \draw [->] (B) to node {$$}(D);
\end{tikzpicture}
\]
The comparison of the corresponding spectral sequences for $N$ and $D^3$ implies that $d_3$ factors through $H_3(\BDiff_0(D^3,\text{rel }\partial);\bZ)$
\[
 \begin{tikzpicture}[node distance=1.5cm, auto]
  \node (A) {$ H_3(\BDiff_0(N,\text{rel }\partial);\bZ)$};
  \node (B) [right of=A, node distance=4.5cm] {$H_2(\overline{\BDiff_0(N,\text{rel }\partial)};\bZ)$};
  \node (C) [below of=A]{$ H_3(\BDiff_0(D^3,\text{rel }\partial);\bZ)$};
  \node (D) [right of=C, node distance=4.5cm]{$  H_2(\overline{\BDiff_0(D^3,\text{rel }\partial)};\bZ).$};
  \draw [->] (A) to node {$d_3$}(B);
    \draw [->] (A) to node {$$}(C);
  \draw [->] (C) to node {$d_3$}(D);
  \draw [->] (B) to node {$\cong$}(D);
\end{tikzpicture}
\]
Note that the fact that the right vertical map is an isomorphism follows from the isomorphism \ref{MT}. On the other hand, by Hatcher's theorem (\cite{hatcher1983proof}) $\Diff_0(D^3,\text{rel }\partial)$ is contractible. Therefore, we have $H_3(\BDiff_0(D^3,\text{rel }\partial);\bZ)=0$ which implies that $d_3$ for $N$ factors through a zero group. Hence, it is a trivial map. This was a special case, that we could argue integrally. Now motivated by this example, let's prove \Cref{H3}.
\begin{proof}[Proof of \Cref{H3}]
Recall $M$ is a manifold whose $\text{dim}(M)$ is not $1$ modulo $4$. As we saw in the above example, it is enough to prove the following $d_3$ differential in the Serre spectral sequence for $\overline{\BDiff_0(M)}\to \BdDiff_0(M)\to \BDiff_0(M)$ is trivial,
\[
d_3\colon H_3(\BDiff_0(M);\bZ)\to H_2(\overline{\BDiff_0(M)};\bZ).
\]
Now since $\BDiff_0(M)$ is simply connected space, the Hurwicz map $\pi_3(\BDiff_0(M))\to H_3(\BDiff_0(M);\bZ)$ is surjective. On the other hand, the long exact sequence of homotopy groups for the fibration $\overline{\BDiff_0(M)}\to \BdDiff_0(M)\to \BDiff_0(M)$ implies that $\pi_3(\BDiff_0(M))\cong \pi_2(\overline{\BDiff_0(M)})$. Hence, to show that the differential $d_3$ is trivial, it is enough to show that the Hurewicz map which is right vertical map in
\[
 \begin{tikzpicture}[node distance=1.5cm, auto]
  \node (A) {$ H_3(\BDiff_0(M);\bZ)$};
  \node (B) [right of=A, node distance=4.5cm] {$H_2(\overline{\BDiff_0(M)};\bZ)$};
  \node (C) [below of=A]{$ \pi_3(\BDiff_0(M))$};
  \node (D) [right of=C, node distance=4.5cm]{$  \pi_2(\overline{\BDiff_0(M)}),$};
  \draw [->] (A) to node {$d_3$}(B);
    \draw [<<-] (A) to node {$$}(C);
  \draw [->] (C) to node {$\cong$}(D);
  \draw [<-] (B) to node {$h$}(D);
\end{tikzpicture}
\]
is trivial.  Now consider the following commutative diagram
\[
 \begin{tikzpicture}[node distance=1.5cm, auto]
  \node (A) {$ \pi_2(\overline{\BDiff_0(M)})$};
  \node (B) [right of=A, node distance=4.5cm] {$H_2(\overline{\BDiff_0(M)};\bZ)$};
    \node (E) [right of=B, node distance=4.5cm] {$H_{n+2}(\overline{\mathrm{B}\Gamma}_n;\bZ)$};
  \node (C) [below of=A]{$ \pi_2(\BdDiff_0(M))=0$};
  \node (D) [right of=C, node distance=4.5cm]{$ H_2(\BdDiff_0(M);\bZ)$};
      \node (F) [right of=D, node distance=4.5cm] {$H_{n+2}(\mathrm{B}\Gamma^+_n;\bZ).$};
  \draw [->] (A) to node {$h$}(B);
    \draw [->] (A) to node {$$}(C);
        \draw [->] (B) to node {$\cong$}(E);
    \draw [->] (D) to node {$p$}(F);
        \draw [->] (E) to node {$i$}(F);
  \draw [->] (C) to node {$$}(D);
  \draw [->] (B) to node {$$}(D);
\end{tikzpicture}
\]
The top right horizontal map is an isomorphism by Mather-Thurston's theorem (\cite[Bottom of page 306]{thurston1974foliations}). And the map $p$ is induced by considering elements in $ H_2(\BdDiff_0(M);\bZ)$ as flat $M$-bundles over a surface so the total space is $(n+2)$-dimensional with a codimension $n$-foliation which gives an element in $H_{n+2}(\mathrm{B}\Gamma^+_n;\bZ).$ Hence, to prove the theorem, it is enough to show that the map $i$ is rationally injective if $n\nequiv 1\bmod 4$. But to show that $i$ is rationally injective, we shall consider the Serre spectral sequence for the fibration 
\[
\overline{\mathrm{B}\Gamma}_n\to \mathrm{B}\Gamma^+_n\to \mathrm{B}\text{GL}^+_n(\bR). 
\]
Since by Thurston's theorem (\cite{thurston1974foliations}) we know that $\overline{\mathrm{B}\Gamma}_n$ is $n+1$ connected, it is enough to show that the differential $H_{n+3}(\mathrm{B}\text{GL}^+_n(\bR);\bQ)\to H_{n+2}(\overline{\mathrm{B}\Gamma}_n;\bQ)$ is trivial. But for $n\nequiv 1\bmod 4$, we know that $H_{n+3}(\mathrm{B}\text{GL}^+_n(\bR);\bQ)$ is trivial. 
\end{proof}
\begin{rem}
To drop the hypothesis $n\nequiv 1\bmod 4$, we need to show that the transgression map $H_{n+3}(\mathrm{B}\text{GL}^+_n(\bR);\bQ)\to H_{n+2}(\overline{\mathrm{B}\Gamma}_n;\bQ)$ is trivial. To determine this map, one could look at the fibration $\text{GL}^+_n(\bR)\xrightarrow{\iota} \overline{\mathrm{B}\Gamma}_n\to \mathrm{B}\Gamma^+_n$ and E.Vogt in (\cite[Problem F.2.1]{MR1271828}) posed the question that whether $\iota$ is nullhomotopic. 
\end{rem}
\subsection{Further discussion for different transverse structures} As we mentioned, the main evidence behind this conjecture \ref{HT} was Gelfand-Fuks computations of continuous Lie algebra cohomology of formal vector fields and also the fact there are no secondary characteristic classes known in degrees lower $2n+1$ for a codimension $n$ foliation. The same line of thought can be applied to foliations with other transverse structures. For example, for the case of having transverse contact structure for a foliation with odd codimenison $n=2k+1$, Feigin (\cite{MR687396}) computed the continuous Lie algebra cohomology of formal contact vector fields and observed that it vanishes at least up to degree $2n$. Similarly, one can formulate the contact version of \Cref{HT}.
\begin{conj*}
Let $(M,\alpha)$ be a contact manifold where $M$ is a manifold of dimension $n=2k+1$ and $\alpha$ is a smooth $1$-form such that $\alpha \wedge (d\alpha)^n$ is
a volume form. The group of orientation preserving $C^{\infty}$-contactomorphisms consists of $C^{\infty}$-diffeomorphisms such that $f^*(\alpha)=\lambda_f\alpha$ where $\lambda_f$ is a non-vanishing positive smooth function on $M$ depending on $f$. Since we are working with orientation preserving automorphisms, we assume that $\lambda_f$ is a positive function. Let $\text{\textnormal{Cont}}_c(M,\alpha)$ denote the group of compactly supported contactomorphisms with induced topology from $C^{\infty}$-diffeomorphisms. Then the natural map
\[
\mathrm{B}\text{\textnormal{Cont}}_c(M,\alpha)^{\delta}\to \mathrm{B}\text{\textnormal{Cont}}_c(M,\alpha),
\]
induces a homology isomorphism up to degree $n$ and a surjection on homology in degree $n+1$. 
\end{conj*}
Another interesting transverse structure for foliations is to have volume preserving holonomies. To formulate a similar question in terms of volume preserving diffeomorphisms, let $M$ be an $n$-dimensional manifold with a fixed volume form $\omega$ and let $\Diff^{\omega}(M)$ denote the group of volume preserving diffeomorphism equipped with the $C^{\infty}$-topology. It is interesting to find the largest homological degree so that up to that degree the map 
\[
\eta\colon\mathrm{B}\Diff^{\omega,\delta}(M)\to \mathrm{B}\Diff^{\omega}(M),
\]
induces a homology isomorphism. Let $\overline{\mathrm{B}\Diff^{\omega}(M)}$ denote the homotopy fiber of $\eta$. McDuff (\cite[\S 2]{MR678355}) showed that when the volume of $\omega$ is infinity the space $\overline{\Diff_c^{\omega}(R^n)}$ has a nontrivial $(n-1)$-th homology.  And in fact, Hurder (\cite{hurder1983global}) proved that the classifying space of Haefliger structures preserving volume form with a trivial normal bundle $\overline{\mathrm{B}\Gamma}_n^{\text{vol}}$ for $n>2$ is not $(n+3)$-connected. Therefore, the best we can expect in the volume preserving case for dimension bigger than $2$ would be
\begin{quest*}
Let $(M,\omega)$ be a pair of an $n$-dimensional manifold $M$ and a volume form $\omega$. Then the map
\[
\eta\colon\mathrm{B}\Diff^{\omega,\delta}(M)\to \mathrm{B}\Diff^{\omega}(M),
\]
induces a homology isomorphism on $H_2(-;\bZ)$ if $\text{dim}(M)>2$. 
\end{quest*}

\section{Flat $G$-bundles vs flat $\Diff_0(G)$-bundles}\label{flat}
Let $G$ be a finite dimensional connected Lie group. A flat $G$-bundle $p\colon E\to M$ over an oriented manifold $M$ gives a cycle in the group homology of $G^{\delta}$. We can consider such flat bundle as a flat $\Diff_0^r(G)$-bundle by extending the holonomy group via the map $G\to \Diff_0^r(G)$, and ask whether it is a nontrivial cycle in group homology of $\Diff_0^r(G)$. In other words, we have the induced map

\[
\alpha_r\colon \mathrm{B}G^{\delta}\to \BDiff^{r,\delta}_0(G),
\]
and we want to study whether $\alpha_r$ is homologically nontrivial. \Cref{HT} as we explained in the introduction suggests that this map might be trivial on integral homology in degrees less than $\text{dim}(G)+1$. 

%We have already showed in the introduction that for a compact or complex semisimple Lie group $G$, the map $\alpha_r$ is trivial on $H_*(-;\bQ)$ for $r=0,1$. 
\begin{proof}[Proof of \Cref{lowreg}]
Recall that Milnor proved (\cite{milnor1983homology}) that for a compact or complex semi-simple Lie group $H_*(\mathrm{B}G^{\delta};\bQ)\to H_*(\mathrm{B}G;\bQ)$ is a trivial map. Given the commutative diagram
\[
 \begin{tikzpicture}[node distance=1.5cm, auto]
  \node (A) {$ H_*(\mathrm{B}G^{\delta};\bQ)$};
  \node (B) [right of=A, node distance=3cm] {$H_*(\BDiff^{r,\delta}_0(G);\bQ)$};
  \node (C) [below of=A]{$H_*(\mathrm{B}G;\bQ)$};
  \node (D) [right of=C, node distance=3cm]{$H_*(\BDiff^r_0(G);\bQ),$};
  \draw [->] (A) to node {$$}(B);
    \draw [->] (A) to node {$$}(C);
  \draw [->] (C) to node {$$}(D);
  \draw [->] (B) to node {$$}(D);
\end{tikzpicture}
\]
and \Cref{HT}, in homological dimension less than $\text{dim}(G)+1$, we expect that the top map is trivial. Note that since \Cref{HT} is known for $r=0,1$ in all homological degree, we already conclude the proof for low regularities.
\end{proof}
So we assume that the regularity $r>1$. Consider the following homotopy commutative diagram 
\[
\begin{tikzpicture}[node distance=1.5cm, auto]
  \node (A) {$\overline{ \mathrm{B}G}$};
  \node (B) [right of=A, , node distance=3cm] {$\overline{\BDiff^{r}_0(G)}$};
  \node (C) [below of=A]{$\mathrm{B}G^{\delta}$};
  \node (D) [right of=C, node distance=3cm]{$\BDiff^{r,\delta}_0(G).$};
  \draw [->] (A) to node {$\overline{\alpha_r}$}(B);
    \draw [->] (A) to node {$$}(C);
  \draw [->] (C) to node {$\alpha_r$}(D);
  \draw [->] (B) to node {$$}(D);
\end{tikzpicture}
\]
As we learned from Mather-Thurston's theory, it is sometimes easier to work with $\overline{\BDiff^{r}_0(G)}$ first. So we work with $\overline{\alpha_r}$ instead and in fact for compact Lie group or complex semisimple groups, studying $\overline{\alpha_r}$  would be enough for our purpose because of the following lemma. 
\begin{lem}\label{collapse}
For a compact Lie group or a complex semisimple group $G$, if the map $\overline{\alpha_r}$ induces a trivial map on $H_k(-;\bQ)$ so does $\alpha_r$.
\end{lem}
\begin{proof}
Dupont (\cite[Theorem 3.1]{MR1286932}) and Brylinski (\cite{MR1247279}) showed that the Serre spectral sequence for the fibration 
\[
G\to \overline{ \mathrm{B}G}\to \mathrm{B}G^{\delta},
\]
collapses rationally if $G$ is a compact or complex semisimple Lie group. Therefore, in particular in these cases, the map
\[
H_*( \overline{ \mathrm{B}G};\bQ)\twoheadrightarrow H_*( \mathrm{B}G^{\delta};\bQ),
\]
is surjective. So if $\overline{\alpha_r}$ induces a trivial map on rational homology in some degree, so will be $\alpha_r$. 
\end{proof}
\begin{thm}
Let $G$ be a real Lie group and $K$ be a maximal compact subgroup. Then, the induced map
\[
\overline{\alpha_r}_*\colon H_*(\overline{ \mathrm{B}G};\bZ)\to H_*(\overline{\BDiff^{r}_0(G)};\bZ),
\]
is a trivial map for $*\leq \text{dim}(G/K)+1$.
\end{thm}
\begin{proof}
In fact, we show that the group $H_*(\overline{\BDiff^{r}_0(G)};\bZ)$ is trivial for $*\leq \text{dim}(G/K)+1$. Since $G$ is parallelizable, by Mather-Thurston's theorem we have the homology isomorphism
\[
\overline{\BDiff^{r}_0(G)}\to \text{Map}(G, \overline{\mathrm{B}\Gamma}_{\text{dim}(G)}).
\]
Since $G$ is homotopy equivalent to its maximal compact subgroup we have $ \text{Map}(G, \overline{\mathrm{B}\Gamma}_{\text{dim}(G)})\xrightarrow{\simeq}  \text{Map}(K, \overline{\mathrm{B}\Gamma}_{\text{dim}(G)})$. On the other hand, by Thurston's theorem we know that $\overline{\mathrm{B}\Gamma}_{\text{dim}(G)}$ is at least $\text{dim}(G)+1$ connected. Therefore, by the fact about homotopy groups of mapping spaces in subsection \ref{back}, $\text{Map}(K, \overline{\mathrm{B}\Gamma}_{\text{dim}(G)})$ is at least $ \text{dim}(G)-\text{dim}(K)+1$-connected. 
\end{proof}
\begin{cor}
For a complex semisimple Lie group $G$, the map
\[
(\alpha_r)_*\colon H_*(\mathrm{B}G^{\delta};\bQ)\to H_*(\BDiff^{r,\delta}_0(G);\bQ),
\]
is trivial for $*\leq \text{dim}(G/K)+1$.
\end{cor}

Let us briefly remark that if we consider the corresponding maps in continuous cohomology theories, one gets a trivial map in all degrees. It is a well known theorem of van Est that the continuous cohomology $H^*_{cont}(G;\bR)$ is isomorphic to the relative Lie algebra cohomology $H^*(\mathfrak{g},\mathfrak{k})$ where $\mathfrak{k}$ is the Lie algebra of maximal compact subgroup. Since $ \overline{\mathrm{B}G}$ is the realization of the etale groupoid given by the action of $G^{\delta}$ on $G$, one can similarly define the continuous (smooth) cohomology $H^*_{cont}( \overline{ \mathrm{B}G};\bR)$ as in \cite{MR660658, MR494071}. And there is a version of  van Est which says that the continuous cohomology $H^*_{cont}( \overline{ \mathrm{B}G};\bR)$ is isomorphic to the Lie algebra cohomology $H^*(\mathfrak{g})$. Hence, we have a commutative diagram
\[
\begin{tikzpicture}[node distance=1.5cm, auto]
  \node (A) {$H^*_{cont}(\overline{ \mathrm{B}G};\bR)$};
  \node (B) [right of=A, , node distance=3cm] {$H^*(\mathfrak{g})$};
  \node (C) [below of=A]{$H^*_{cont}(G;\bR)$};
  \node (D) [right of=C, node distance=3cm]{$H^*(\mathfrak{g},\mathfrak{k}).$};
  \draw [<-] (A) to node {$\cong$}(B);
    \draw [<-] (A) to node {$$}(C);
  \draw [<-] (C) to node {$\cong$}(D);
  \draw [<-] (B) to node {$$}(D);
\end{tikzpicture}
\]
Similarly Brown-Szczarba (\cite{MR1295578}) proved that $H^*_{cont}(\overline{\BDiff^{r}_0(M)};\bR)$ is isomorphic to the continuous Lie algebra cohomology (aka Gelfand-Fuks cohomology) $H^*(\text{Vect}(M))$. So $\overline{\alpha_r}$ on the level of the continuous cohomology is the map
\[
H^*(\text{Vect}(G))\to H^*(\mathfrak{g}).
\]
Interestingly, this map is trivial in all degrees. Because, it is consequence of Bott-Segal's theorem (\cite{bott1977cohomology}) that $H^*(\text{Vect}(M))$ is trivial in degrees less than $\text{dim}(M)+1$ and $H^*(\mathfrak{g})$ is trivial by definition for degrees above $\text{dim}(G)$. So as a consequence, for a semisimple Lie group $G$,  the map between smooth group cohomologies
\[
H^*_{cont}(\Diff_0^r(G);\bR)\to H^*_{cont}(G;\bR),
\]
is trivial in all degrees.

Now back to group homology with integer coefficients, recall that we know that $H_1(\BDiff^{r,\delta}_0(G);\bZ)=0$ for all $r\neq \text{dim}(G)+1$. Hence, the first nontrivial homological degree that $\alpha_r$ could be nontrivial for $r\neq \text{dim}(G)+1$ is 
\[
(\alpha_r)_*\colon H_2(\mathrm{B}G^{\delta};\bZ)\to H_2(\BDiff^{r,\delta}_0(G);\bZ).
\]
Sah-Wagoner (\cite{MR646087}) proved that for any connected Lie group $G$, the second group homology $H_2(\mathrm{B}G^{\delta};\bZ)$ has a quotient group equal to a $\bQ$-vector space of dimension equal to the continuum. We consider the case where $G$ is abelian or $G=\text{SU}_2$.
\begin{thm}Let $G$ be a finite dimensional abelian connected Lie group or let it be $\text{SU}_2$ and $r\neq \text{dim}(G)+1$, then the induced map 
\[
H_2(\mathrm{B}G^{\delta};\bQ)\to H_2(\BDiff^{r,\delta}_0(G);\bQ)
\]
is trivial.
\end{thm}
\begin{proof}
First let us consider the abelian case. If $G$ is not compact, then $G\cong \bR^k\times T^n$ for some $k>0$ and the group homology $H_*(\mathrm{B}(\bR^k)^{\delta}\times \mathrm{B}(T^n)^{\delta};\bQ)\cong H_*(\mathrm{B}(\bR^k)^{\delta};\bQ)\times H_*(\mathrm{B}(T^n)^{\delta};\bQ)$. We show that cycles in $H_*(\mathrm{B}(\bR^k)^{\delta};\bQ)$ map trivially into $H_2(\BDiff^{r,\delta}_0(G);\bQ)$. Note that $\mathrm{B}(\bR^k)^{\delta}\to \BDiff^{r,\delta}_0(G)$ factors as follows
\[
\mathrm{B}(\bR^k)^{\delta}\xrightarrow{\beta} \mathrm{B}(\text{Aff}(\bR^k))^{\delta}\to \BDiff^{r,\delta}_0(G). 
\]
There is a trick that apparently goes back to Quillen that the group homomorphism $\text{Aff}(\bR^k)\to \text{GL}_k(\bR)$ induces an isomorphism on rational group homology (see \cite[Lemma 4]{de1983acyclic}). Therefore, the map $\beta$ induces a trivial map on $H_*(-;\bQ)$. 

So we assume that $G=T^k$.  We shall first consider the case $k=1$. Geometrically, any $2$-cycle in $H_2(S^1;\bZ)$ is represented by a flat $S^1$-bundle over the $2$-torus. Equivalently, on the total space which is diffeomorphic to $T^3$ we have  a foliation transverse to the $S^1$-fibers whose holonomy is given by a representation $\rho\colon\pi_1(T^2)\to \text{Rot}(S^1)$ of fundamental group of the base into the rotations of $S^1$. Such a foliation is given by the integrable form $\omega=dz-(a.dx+b.dy)$ where $z$ is the coordinate of the fiber and $x$ and $y$ are the coordinates of the base. But not only $\omega$ is integrable but also it is closed. We learned from \cite[Page 145]{moussu1974relations} that two codimension $1$-foliations on $M$ that are defined by closed $1$-forms $\omega_1$ and $\omega_2$ are in fact concordant. Because we can consider the foliation on $M\times [0,1]$ that is defined by the integrable $1$-form $dt+f_0(t).\omega_1+f_1(t).\omega_1$ where $f_i$ are smooth functions on the real line such that  $\text{Supp}(f_0(t))\subset [0,1/4]$ and $\text{Supp}(f_1(t))\subset [3/4,1]$. So the foliation defined by $\omega$ is concordant to the foliation defined by $\omega'=dz$. But the foliation defined by $\omega'$ is foliated cobordant to zero since it is trivial horizontal foliation on $S^1\times T^2$ which bounds the horizontal foliation on $S^1\times (D^2\times S^1)$. Hence, the foliation defined by $\omega$ gives a trivial cycle in $H_3(\mathrm{B}\Gamma^r_1;\bZ)$. 

On the other hand, by Mather-Thurston's theorem $H_2(\BDiff_0^{r,\delta}(S^1);\bZ)\cong \bZ\oplus H_3(\mathrm{B}\Gamma^r_1;\bZ)$ where the $\bZ$ summand is detected by the Euler class and $H_3(\mathrm{B}\Gamma^r_1;\bZ)\cong H_2(\overline{\BDiff_0^{r,\delta}(S^1)};\bZ)$. But the Euler class of the foliation defined by $\omega$ on $S^1\times T^2$ is trivial and by the above argument it is also a trivial cycle in $H_3(\mathrm{B}\Gamma^r_1;\bZ)$. Hence, the map 
\[
H_2(\mathrm{B}(S^1)^{\delta};\bZ)\to H_2(\BDiff_0^{r,\delta}(S^1);\bZ),
\]
is trivial. 

Now for the case $G=T^n$, recall from the proof of \Cref{collapse} that $H_2(\overline{\mathrm{B}(T^n)};\bQ)\to H_2(\mathrm{B}(T^n)^{\delta};\bQ)$ is surjective. So instead we shall prove that 
\[
H_2(\overline{\mathrm{B}(T^n)};\bQ)\to H_2(\overline{\BDiff_0^r(T^n)};\bQ),
\]
is a trivial map. Recall that by Mather-Thurston's theorem and the connectivity of $\overline{\mathrm{B}\Gamma}^r_{n}$, for $r\neq \text{dim}(G)+1$, so similar to isomorphisms in (\ref{MT}), we have 
\[
H_2(\overline{\BDiff_0^r(T^n)};\bZ)\cong H_2(\text{Map}(T^n,\overline{\mathrm{B}\Gamma}^r_{n});\bZ)\cong H_{n+2}(\overline{\mathrm{B}\Gamma}^r_{n};\bZ).
\]
Therefore, it is enough to show that any $2$-cycle in $H_2(\overline{\mathrm{B}(T^n)};\bQ)$ which is a trivialized flat $T^n$-bundle over $T^2$ maps trivially into $H_{n+2}(\overline{\mathrm{B}\Gamma}^r_{n};\bQ)$, which in turn follows if we  show that such flat $T^n$-bundles over $T^2$ are trivial in the foliated cobordism group. Since the holonomy group $\rho(\bZ^2)<T^n$, the foliation is given similarly to the previous case by the Pfaffian system
\begin{align*}
\omega_i=dz_i- (a_i.dx+b_i.dy) & = 0\text{ for all }i,
\end{align*}
where $z_i$ are coordinates of $T^n$ fiber and $x$ and $y$ are coordinates of $T^2$ base. Again since these are all closed one forms, this foliation is foliated cobordant to the horizontal foliation on $T^n\times T^2$ given by $dz_i=0$. But the horizontal foliation is trivial in foliated cobordism group so the image of this $2$-cycle is trivial in $H_2(\overline{\BDiff_0^r(T^n)};\bZ)$.

Now let $G=\text{SU}_2$. Mather proved in a letter to Sah (\cite{Matherletter}, see also \cite{alperin1979k2}) that 
\[
H_2(\mathrm{B}(S^1)^{\delta};\bZ)\to H_2(\mathrm{B}(\text{SU}_2)^{\delta};\bZ),
\]
is surjective. So any $2$-cycle in $H_2(\mathrm{B}(\text{SU}_2)^{\delta};\bZ)$ can be represented by a flat $\text{SU}_2$-bundle over $T^2$ whose holonomy group $\rho\colon \bZ^2\to \text{SU}_2$ lies in the maximal torus $S^1<\text{SU}_2$. In particular, the holonomy lies in a subgroup generated on a one-parameter subgroup given by the flow of 
\begin{equation*}
u_3=
\begin{bmatrix}
i& 0 \\
0 & -i 
\end{bmatrix}
\end{equation*}
in the Lie algebra $\mathfrak{su}(2)$.  Let $u_1$ and $u_2$ be the other generators of  $\mathfrak{su}(2)$. So any element in $H_2(\mathrm{B}(\text{SU}_2)^{\delta};\bZ)$ can be represented by a foliation on $\text{SU}_2\times T^2$ that is defined by the Pfaffian system 
\begin{align*}
\omega_3=dz_3- (a.dx+b.dy) & = 0\\
\omega_2=dz_2& =0\\
\omega_1=dz_1& =0
\end{align*}
where $dz_i$ are one forms dual to $u_i$ in the Lie algebra $\mathfrak{su}(2)$.  Since this Pfaffian system is given by the vanishing of closed forms, by a similar argument for codimension $1$-foliations (\cite[Page 145]{moussu1974relations}), it is foliated cobordant to a foliation given by $dz_i=0$ for all $i$ which is the horizontal codimension $3$ foliation on $\text{SU}_2\times T^2$. Therefore, it is trivial in the foliated cobordism group. 
\end{proof}
The case of $G=\text{\textnormal{SU}}_2$ is particularly interesting because its third group homology is also known and there are non-abelian cycles in degree $3$. In particular, it is known (\cite[Corollary 9.19]{MR1832859}) that $H_3(\mathrm{B}\text{\textnormal{SU}}_2^{\delta};\bZ)$ is isomorphic to $\bQ/\bZ$ plus a nontrivial $\bQ$ vector space which is subspace of the scissors congruence $\mathcal{P}_{\bC}$.  The summand $\bQ/\bZ$ is again generated by abelian cycles but the other summand that is detected by the scissors congruence group is not generated by abelian cycles and \Cref{HT} predicts that these cycles should be trivial too in $H_3(\BDiff^{r,\delta}_0(S^3);\bZ)$. In this direction, Reznikov (\cite[Theorem 6.6]{MR1694899}) proved that Chern-Simons classes in $H^3(\mathrm{B}\text{\textnormal{SO}}_4(\bR)^{\delta};\bR/\bZ)$ lift to $H^3(\BDiff^{vol,\delta}(S^3);\bR/\bZ)$ where the group $\Diff^{vol,\delta}_0(S^3)$ is the volume preserving diffeomorphisms of $S^3$ made discrete. However, because of \Cref{HT}, we expect that they cannot be further lifted to $H^3(\BDiff^{\delta}_0(S^3);\bR/\bZ)$.

\begin{conj} Chern-Simons classes are in the cokernel of the map 
\[
H^3(\BDiff^{\delta}_0(S^3);\bR/\bZ)\to H^3(\mathrm{B}\text{\textnormal{SO}}_4(\bR)^{\delta};\bR/\bZ).
\]
\end{conj}
 \bibliographystyle{alpha}
\bibliography{reference}
\end{document}